\documentclass[12pt,twoside]{article}
\usepackage{amsmath, amssymb, amsthm}
\newtheorem{theorem}{Theorem}
\newtheorem{corollary}{Corollary}
\newtheorem{lemma}{Lemma}
\newtheorem{assertion}{Assertion}
\newtheorem{definition}{Definition}

\begin{document}
\vspace*{-1.0cm}\noindent \copyright
 Journal of Technical University at Plovdiv\\[-0.0mm]\
\ Fundamental Sciences and Applications, Vol. 11, 2006\\[-0.0mm]
\textit{Series A-Pure and Applied Mathematics}\\[-0.0mm]
\ Bulgaria, ISSN 1310-8271\\[+1.2cm]
\font\fourteenb=cmb10 at 12pt
\begin{center}

   {\bf \LARGE A Connection on Manifolds with a Nilpotent Structure \\ \ \\ \large Asen Hristov}
\end{center}

\

\

 \footnotetext{ Received September 15, 2006 }

\begin{abstract}
All the connections, pure toward the nilpotent stucture, are found. Examples of manifolds, for which the
curvature tensor is pure or hybridous, are given. For a manifold of B-type a necessary and sufficient condition
for purity of the curvature tensor is proved. It is verified that the conformal change of the metric of a
B-manifold does not retain its puruty.

\end{abstract}
\

We suppose that $B$ is a submanifold of $E$, where $\dim B = n$, and
$\dim E = m+n$. We assume that
\[\delta : E \rightarrow B\]
is a submersion in $E$. If point $p$ belongs to the base $B$ then
the set of points
\[\delta^{(-1)} (p) \subset E\]
are a layer over it. A priori we suppose a local triviality of
$\delta$. By $TB$ we denote the tangential differentiation on the
base $B$. If we designate by
\[( z^{i},z^{n+i},z^{2n+a} ),\  i
=1,2,\dots,n; \ a =1,2,\dots,m \ ,\] then $z^{i}$ are coordinates of
a point from $B$, $z^{n+i}$ are coordinates of a point in a layer
from $TB ( \pi : TB \rightarrow B $, where $\pi^{-1} (p)$ is a
tangential layer over a point $p \in B )$, and $z^{2n+a}$ are
coordinates of a point from $\sigma^{-1} (p)$. We can interpret the
variables
\[ \big( z^{i}, z^{n+i}, z^{2n+a} \big)\]
as local coordinates of a point from a $2n+m$ - dimensional manifold
$\mathfrak{M}$ provided that these coordinates are replaced by the
following rule \cite{2}:

\[ z^{i} = \varphi^{i} (\overline{z}^{1}, \overline{z}^{2}, \dots , \overline{z}^{n}) \]

\begin{equation} z^{n+i} = \sum_{k} \frac{\displaystyle\partial\varphi^{i} (\overline{z}^{1}, \overline{z}^{2}, \dots , \overline{z}^{n})}{\displaystyle\partial \overline{z}^{k}} . \overline{z}^{n+k} \end{equation}

\[ z^{2n+a} = \Theta^{a} (\overline{z}^{1}, \overline{z}^{2}, \dots , \overline{z}^{n}; \ \overline{z}^{2n+1}, \overline{z}^{2n+2}, \dots , \overline{z}^{2n+m}) \]

It is proved in \cite{2} that $\mathfrak{M}$ tolerates an
integratable nilpotent structure $f$, called semitangential by the
author of this paper, which commutates with the Jacobian of (1).
With regard to a suitable local basis on $\mathfrak{M}$, the matrix
of $f$ has the following form
\[ \big( f^{\alpha}_{\beta} \big) =
\left( \begin{array}{c|c|c}
\ & \ & \ \\
\hline
E & \ & \ \\
\hline \ \ & \ & \end{array} \right) \ ,\]
E  is a single matrix
from row $n$, and the blank blocks are zero blocks.

We assume that $\nabla = \mathfrak{g}^{\sigma}_{\alpha \beta}$ is a
torsionless connection on $\mathfrak{M}$ with connection
coefficients $\mathfrak{g}^{\sigma}_{\alpha \beta}$. As regards
$\nabla$, we impose a purity condition of the coefficients toward
$f$ as well:

\[ \mathfrak{g}^{\lambda}_{\beta \sigma} f^{\alpha}_{\beta} = \mathfrak{g}^{\alpha}_{\beta \lambda} f^{\lambda}_{\sigma} = \mathfrak{g}^{\alpha}_{\lambda \sigma} f^{\lambda}_{\beta} \ , \ \ \ \ \ \alpha, \beta, \sigma, \lambda, \dots = 1, 2, \dots , m+2n \]

A solution to the above condition for purity is

\begin{equation} \mathfrak{g}^{k}_{is} = \mathfrak{g}^{n+k}_{i, n+s} = \mathfrak{g}^{n+k}_{n+i,s}  \ , \  \mathfrak{g}^{n+k}_{is} \ \ \mbox{and} \ \ \mathfrak{g}^{2n+a}_{\beta \sigma}, \mathfrak{g}^{\sigma}_{\beta, 2n+a}
\end{equation}

In the special case $m=0$, i.e. when $\mathfrak{M}$ is the
tangential differentiation of $B$, the coefficients of the second
series of solution (2) are
\[\mathfrak{g}^{2n+a}_{\beta \sigma} =
\mathfrak{g}^{\sigma}_{\beta, 2n+a} = 0 \ . \]

In this case, if $\overset{\circ}{\nabla} = (\mathfrak{g}^{h}_{ik})$
is a torsionless connection on the base, then we denote
$u^{i}=z^{i}, \ y^{i}=z^{n+i}$ so therefore

\[ \mathfrak{g}^{h}_{ik} = \mathfrak{g}^{h}_{ik}, \ \mathfrak{g}^{n}_{i, n+k} = \mathfrak{g}^{h}_{n+i,k} = \mathfrak{g}^{n}_{n+i, n+k} = \mathfrak{g}^{n+n}_{n+i, n+k} = 0 \ ,\]

\[ \mathfrak{g}^{n+n}_{ik} = z^{s} \frac{\partial \mathfrak{g}^{h}_{ik}}{\partial u^{s}} \ ,\]

\[ \mathfrak{g}^{n+h}_{i,n+k} = \mathfrak{g}^{n+h}_{n+i,k} = \mathfrak{g}^{h}_{ik} \ .\]

Thus $\mathfrak{g}^{\alpha}_{\beta \sigma}$ coincide with the
coefficients of the complete lift of the connection
$\overset{\circ}{\nabla}$.

\begin{lemma} If a connection has pure coefficients with respect to one
upper and one lower index, and f is covariantly constant, then f is
an integratable structure. \end{lemma}

\begin{proof}
From the condition

\[ \nabla_{\sigma} f^{\alpha}_{\beta} = \partial_{\sigma}f^{\alpha}_{\beta} - \mathfrak{g}^{\alpha}_{\sigma \lambda} f^{\lambda}_{\beta} + \mathfrak{g}^{\nu}_{\sigma \beta} f^{\alpha}_{\nu} = 0 \]

follows that

\[ \nabla_{\sigma} f^{\alpha}_{\beta} = \partial_{\sigma}f^{\alpha}_{\beta} = 0 \ ,  \]

which shows that $f^{\alpha}_{\beta} = const$.

\end{proof}

\begin{corollary} If f is an integratable structure and $\nabla f = 0$, then we
have purity of the coefficients $\mathfrak{g}^{\alpha}_{\beta
\sigma}$ with respect to one upper and one lower
index.\end{corollary}

\begin{lemma} If $f$ is integratable and $R^{\sigma}_{\alpha \beta \gamma}$ are the corresponding
components of the curvature tensor for $\nabla$, then the condition
$\nabla f = 0$ is sufficient for $R^{\sigma}_{\alpha \beta \gamma}$
to be pure with respect to one upper and one lower index.
\end{lemma}

\begin{proof} From the Ricci identity

\[ \nabla_{\alpha} \nabla_{\beta} f^{\gamma}_{\sigma} - \nabla_{\beta} \nabla_{\alpha} f^{\gamma}_{\sigma} = R^{\sigma}_{\alpha \beta \gamma} f^{\gamma}_{\sigma} - R^{\sigma}_{\alpha \beta \gamma} f^{\gamma}_{\lambda} \]

the proof is obvious. \end{proof}

The purity of the Riemannian curvature tensor with respect to an
upper and a lower index does not automatically mean purity with
respect to two lower indices. An example of this are Riemannian
manifolds, the functional tensor of which conforms in a suitable way
with the semitangential structure $f$.

\

\textbf{Example 1.} The Riemannian manifold $\mathfrak{M} (g,f)$
with a metric tensor $g = (g_{\alpha \beta})$ and an assigned
integratiable nilpotent tensor field $f$ of type (1,1) is called a
manifold of $B$-type, if for every two vector fields $x$ and $y$
there follows:

\[ g(fx,y) = g(x,fy) \ . \]

We note down that the last condition is equivalent to the purity of
$g_{\alpha \beta}$, where $\tilde{g} = (g_{\alpha \lambda}
f^{\lambda}_{\beta})$ is a symmetrical tensor field, and under
certain conditions $R^{\sigma}_{\alpha \beta \gamma}$ is pure with
respect to all indices.

Detailed information about this type of manifolds and research work,
connected with them, can be found in the studies of E. V. Pavlov.

\

\textbf{Example 2.} Riemannian manifolds $\mathfrak{M} (g,f)$ with a
metric tensor $g = g_{\alpha \beta}$ and an assigned integratiable
nilpotent tensor field $f$ of type (1,1) is called a manifold of
$K\ddot{a}hler$ type if for every two vector fields $x$ and $y$
there follows:

\[ g(x, fy) = - g(fx, y) \]

The last equation is called a condition for the hybridity of $g$
toward $f$. In this example the tensor field $\tilde{g} = (g_{\alpha
\lambda} f^{\lambda}_{\beta} = \tilde{g}_{\beta \alpha})$ is
antisymmetric, $R^{\sigma}_{\alpha \beta \gamma}$ is not pure with
respect to all indices.

Particular examples of this type of $K\ddot{a}hler$ manifolds can be
found in (\cite{3}, p. 137).

\begin{assertion} If $\mathfrak{M} (g,f)$ is a manifold from type $B$, then the
components $R^{\gamma}_{\sigma \alpha \beta}$ and $R_{\sigma \alpha
\beta \gamma}$ are pure with respect to the two indices $(\beta ,
\gamma)$. If $\mathfrak{M}(g,f)$ is of $K\ddot{a}hler$ type, then
$R_{\sigma \beta \alpha \gamma}$ and $R^{\tau}_{\sigma \beta
\alpha}$ are hybrid with respect to $(\alpha, \gamma)$ and pure with
respect to $(\alpha, \tau)$ respectively.
\end{assertion}

\begin{proof} The condition in both cases follows from Lemma 2. We have

\[ R^{\lambda}_{\alpha \beta \gamma} \ . \ f^{\sigma}_{\alpha} = R^{\sigma}_{\alpha \beta \lambda} \ . \ f^{\lambda}_{\gamma}  \]

Hence follows that \[ R^{\lambda}_{\alpha \beta \gamma} . F_{\lambda
\sigma} = R^{\sigma}_{\alpha \beta \lambda} . f^{\lambda}_{\gamma} \
,\]

where $F_{\lambda \sigma} = g_{\sigma \tau} f^{\tau}_{\lambda}$, and
taking the antisymmetrization of $F_{\lambda \sigma}$ into
consideration, we obtain

\[ - R^{\lambda}_{\alpha \beta \gamma} . f^{\lambda}_{\sigma} = R^{\lambda}_{\alpha \beta \gamma} . f^{\lambda}_{\gamma} \ .\]

When $\mathfrak{M}(g,f)$ is of type $B$, the proof is analogical.
\end{proof}

\

In connection with the $B$-type manifolds, two important theorems
need to be mentioned, which are proved in \cite{4} and \cite{5}
respectively.

\

\begin{theorem} If the components $R_{\sigma \alpha \beta \lambda}$ of the curvature tensor of a $B$-type
manifold are pure with respect to two indices, then they are pure
with respect to all double indices. \end{theorem}

\begin{theorem} The components $R_{\sigma \alpha \beta \lambda}$ of the curvature tensor of a $B$-type
manifold are pure then and only then when the partial derivatives of
the metric tensor are pure. \end{theorem}

\

\textbf{Corollary from Assertion 1.} Tensor $G$, defined by means of
the equation
\[ G(x,y,v,w) = g(x,v)g(y,w) - g(x,w)g(y,v)\]
is curvature-related for $\mathfrak{M}(g,f)$. At that the tensor
$\overset{\ast}{G}$, for which
\[ \overset{\ast}{G}(x,y,v,w) = G(x, fy, v, fw) =
-g(x,fw)g(fy,v)\] is pure or hybrid depending on $\mathfrak{M}
(g,f)$. In case of
\[x \neq \ker{f}, \ G(x, fx, y, fy) = -
[g(x,fy)]^{2} \ ,\]

from where follows:

\

a) If $\mathfrak{M} (g,f)$ is of type $B$, then $g(x, fx) \neq 0$.
In this case, on the basis of Assertion 1 and Theorem 1, it is
possible for $\mathfrak{M} (g,f)$ to be flat on account of the fact
that the holomorphic curvature in the direction of $\{x,fx\}$ is
zero.

\

b) If $\mathfrak{M} (g,f)$ is of $K\ddot{a}hler$ type. Now $g(x,fx)
= 0$. Therefore the holomorphic curvature in the direction of
$\{x,fx\}$ is indefinite.

\begin{theorem} If $\mathfrak{M} (g,f)$ is a $B$-type manifold, then the Riemannian
connection, originating from $g$, possesses pure connection
coefficients with respect to all indices then and only then, when
the partial derivatives
\[\partial_{\sigma} g_{\alpha \beta} =
\frac{\partial}{\partial z^{\sigma}} g_{\alpha \beta}\]
are pure
towards $f$.
\end{theorem}

\begin{proof} We need to note down in advance that the symmetrical tensor
field
\[ \tilde{g} = (\tilde{g}_{\alpha \beta} = g_{\lambda \beta} f^{\lambda}_{\alpha}) \]
is transferred parallely towards the Riemannian connection,
originating from $g$. Because of $f^{\alpha}_{\beta} = const$ we
have

\[ \nabla_{\sigma} \tilde{g}_{\alpha \beta} = \partial_{\sigma} \tilde{g}_{\alpha \beta} - \mathfrak{g}_{\sigma \alpha}^{\lambda} \tilde{g}_{\lambda \beta} - \mathfrak{g}_{\sigma \beta}^{\tau} \tilde{g}_{\alpha \tau} = 0 \]

Here we used the corollary from Lemma 1.

\

a)  Let $\partial_{\sigma} g_{\alpha \beta}$ be pure toward $f$. In
this case from the condition

\begin{equation} 2 \mathfrak{g}^{\sigma}_{\alpha \beta} = g^{\sigma \lambda} (\partial_{\alpha} g_{\beta \lambda} + \partial_{\beta} g_{\lambda \alpha} - \partial_{\lambda} g_{\alpha \beta}) \end{equation}

and the purity of $g^{\sigma \lambda}$, resulting from the equations

\begin{eqnarray}
g^{\lambda \sigma} g_{\lambda \gamma} &=&
\delta^{\sigma}_{\gamma} \Rightarrow \nonumber \\ \ \nonumber \\
& \Rightarrow & g^{\lambda \sigma} g_{\lambda \gamma}
f^{\nu}_{\sigma} = f^{\nu}_{\gamma} \Rightarrow \nonumber \\ \
\nonumber \\ &\Rightarrow& g^{\lambda \sigma}
\delta^{\beta}_{\lambda} f^{\nu}_{\sigma} = f^{\nu}_{\gamma}
g^{\gamma \beta} \Rightarrow \nonumber \\ \ \nonumber \\
&\Rightarrow& g^{\beta \sigma} f^{\nu}_{\sigma} = g^{\sigma \beta}
f^{\nu}_{\gamma} \ \ \  \mbox{(the purity of $g^{\alpha \beta}$)} \
, \nonumber
\end{eqnarray}

there follows the purity of the connection coefficients with respect
to all indices.

\

b)  We assume that $\mathfrak{g}^{\sigma}_{\alpha \beta}$ are pure
with respect to all indices. In this case from $\nabla_{\sigma}
g_{\alpha \beta} = 0$ follow the conditions

\[ \partial_{\sigma} \tilde{g}_{\alpha \beta} - \mathfrak{g}^{\lambda}_{\sigma \alpha} \tilde{g}_{\lambda \beta} - \mathfrak{g}^{\nu}_{\sigma \beta} \tilde{g}_{\alpha \nu} = 0 \]

and

\[ f^{\tau}_{\sigma}\partial_{\tau} g_{\alpha \beta} - \mathfrak{g}^{\lambda}_{\sigma \alpha} \tilde{g}_{\lambda \beta} - \mathfrak{g}^{\nu}_{\sigma \beta} \tilde{g}_{\alpha \nu} = 0 \ . \]

By means of their term-by-term subtraction we obtain

\[ f^{\tau}_{\sigma}(\partial_{\tau} g_{\alpha \beta}) = \partial_{\sigma} ( f^{\lambda}_{\alpha} g_{\lambda \beta} ) \ .\]

\end{proof}

\begin{definition} If for $\mathfrak{M} (g,f)$ the partial derivatives of the components of $g$ are
pure, we say that $\mathfrak{M} (g,f)$ is a $B$-manifold.
\end{definition}

\begin{theorem} If $\mathfrak{M} (g,f)$ is a $B$-manifold and $\det (g + \tilde{g}) \neq 0$, then the metrics $g$ and
$g + \tilde{g}$ originate one and same Riemannian connection.
\end{theorem}

\begin{proof}
Let us take into consideration that the matrices $(g_{\alpha \beta}
+ \tilde{g}_{\alpha \beta})$ and $(g^{\alpha \beta} -
\tilde{g}^{\alpha \beta})$ are mutually inverse. Here we have
denoted $\tilde{g}^{\alpha \beta} = g^{\alpha \lambda}
f^{\beta}_{\lambda}$ and used the purity of $g_{\alpha \beta}$,
which was proved in the previous theorem. If
$\overline{\mathfrak{g}}^{\sigma}_{\alpha \beta}$ are the
coefficients of the Riemannian connection, originating from $g +
\tilde{g}$, we apply formula (3) but in reference to the metric $g +
\tilde{g}$. In the course of the calculations we should keep in mind
that the objects $\partial_{\sigma} g_{\alpha \beta}$ are pure. Thus
there follows that $\overline{\mathfrak{g}}^{\sigma}_{\alpha \beta}
= \mathfrak{g}^{\sigma}_{\alpha \beta}$.
\end{proof}

\

The theorem proved holds true in the most common case. In the
special case where $m = 0$, i.e. when $\mathfrak{M} (g,f)$ is a
tangential differentiation on base $B$, a similar theorem is proved
for the lifts of the base metric in (\cite{7}, p. 149). There $g +
\tilde{g}$ is designated as a metric I + II. On the other hand, the
change $g \rightarrow g + \tilde{g}$ is the simplest CH-change. On
condition that $f^{2} = I$ ($I$ is the identical transformation),
the CH-change is studied in detail in \cite{4}. To the question
whether it is the conformal change that needs to be investigated, or
the generalization and the CH-change for $\mathfrak{M} (g,f)$ with
$f^{2} = 0$, we obtain an answer by means of

\begin{theorem}
If $h(z^{1}, \dots \ , z^{m+2n})$ is a random manifold, and
$\mathfrak{M} (g,f)$ is a $B$-manifold, then $\mathfrak{M} (hg,f)$
is a $B$-manifold then and only then, when $h = const$.
\end{theorem}

\begin{proof}
If $\mathfrak{M} (hg,f)$ is a $B$-manifold, there follows that

\[ \partial_{\lambda} (hg_{\alpha \beta}) f^{\lambda}_{\sigma} = \partial_{\sigma}(hg_{\lambda \beta} f^{\lambda}_{\alpha}) \ . \]

Hence we obtain

\[ f_{\sigma}^{\lambda} h_{\lambda} \delta^{\beta}_{\alpha} = h_{\sigma} f^{\beta}_{\alpha} \ , \ \ \  (h_{\lambda} = \frac{\partial h}{\partial z^{\lambda}}) \]

This equation holds true for all values of the indices. For $\beta =
n+i\ , \alpha = k \ , f^{n+i}_{k} = \delta^{i}_{k}$ in particular,
from the specified equation we have

\[ f_{\sigma}^{\lambda} h_{\lambda} \delta^{n+i}_{k} = h_{\sigma} \delta^{i}_{k} \]

or $h_{\sigma} = 0 $ for every $\sigma$.

In the opposite case, where $h = const $, it is obvious that
$\mathfrak{M} (hg,f)$ is a $B$-manifold.
\end{proof}

\textbf{Geodesic and holomorphic plane curves.} If we assume that on
the manifold $\mathfrak{M} (f)$, which is provided with a pure
connection with respect to $f$, the curve $c: z^{\alpha} =
z^{\alpha} (+)$ is geodesic, after changing the parameter $t : t =
h(q)$, the equation of $c$

\[ \frac{\delta \dot{z}^{\alpha}}{dt} = \frac{d^{2} z^{\alpha}}{dt^{2}} + \gamma^{\alpha}_{\lambda \beta} \frac{d z^{\lambda}}{dt} . \frac{d^{2} z^{\beta}}{dt} = 0 \]

($ \dot{z}^{\alpha} = \frac{d z^{\alpha}}{dt}$ is the covariant
differentiation toward the connection, which has the
$\gamma^{\alpha}_{\sigma \beta}$ coefficients) is equivalent to

\[ \frac{\delta}{dq}(\frac{d z^{\alpha}}{dq}) = \frac{dh}{dq}.\frac{dz^{\alpha}}{dq} \ . \]

It is possible to transfer the field $\upsilon =
(\upsilon^{\alpha})$ parallely with respect to $c$

\

\[ \frac{\delta \upsilon^{\alpha}}{dq} =  \frac{d \upsilon^{\alpha}}{dq} + \gamma^{\alpha}_{\lambda \beta} \dot{z}^{\lambda} \frac{d \upsilon^{\beta}}{dt} = 0 \]

\

After changing $\upsilon^{\alpha} \rightarrow \lambda (t)
\upsilon^{\alpha}$ we have

\

\[ \frac{\delta}{dt}(\lambda(t)\upsilon^{\lambda}) = \lambda^{\prime} (t) \upsilon^{\alpha} \]

\

In both of the examples given we will say that the directions
$\frac{d z^{\alpha}}{dt}$ and $\lambda \upsilon^{\alpha}$ are
transferred parallely with respect to $c$.

\

We say that the curve $z^{\alpha} (t)$ from $\mathfrak{M} (f)$ is a
holomorphic plane curve (a PH-curve) for $\mathfrak{M} (f)$, if
$z^{\alpha} (t)$ are the solutions to the differential equation

\[ \frac{\delta \dot{z}^{\alpha}}{dt} = \frac{d^{2} z^{\alpha}}{dt^{2}} + \gamma^{\alpha}_{\lambda \beta} \frac{d z^{\lambda}}{dt} . \frac{d^{2} z^{\beta}}{dt} = a(t) \frac{d z^{\lambda}}{dt} + b(t)(f^{\lambda}_{\nu} \frac{dz^{\nu}}{dt}) \ . \]

Here $a(t)$ and $b(t)$ are functions. We will adopt the designation

\[ \frac{\widetilde{dz^{\alpha}}}{dt}= f^{\lambda}_{\nu} \frac{dz^{\nu}}{dt} \ . \]

There exists a special case, in which the geodesic curves and the
PH-curves coincide. That is when $\mathfrak{M}(f)$ is
projective-Euclidean or PH-Euclidean. This peculiarity is
illustrated best about the three-dimensional projective space
$B_{3}$ (an extension of the corresponding Euclidian one). In
$B_{3}$ the above curves coincide with the absolute straight line
$\omega$, which has a common point with each straight line from the
absolutely congruent straight lines \cite{8}. The absolute straight
line $\omega$ is the set of all points belonging to $\ker f$.

\begin{assertion}
If $z^{\alpha} (t)$ is geodesic for $\mathfrak{M} (f)$ and
$\dot{z}^{\alpha} \overline{\in} \ker f$, then $z^{\alpha} (t)$ is a
PH-curve.
\end{assertion}

\begin{proof}
Let us recall that $f$ is covariantly constant and the objects
$\gamma^{\alpha}_{\lambda \beta}$ are pure with regard to $f$. In
this case, from the differential equation of the geodesic curve
there follows

\[ \frac{\delta}{dt} \tilde{\dot{z}}^{\alpha} = 0 \ . \]

Let us consider the linear combination $\upsilon^{\alpha} = h(t)
\dot{z}^{\alpha} + l(t)\tilde{\dot{z}}^{\alpha}$ for certain
functions $h(t)$ and $l(t)$. For the vector field $\upsilon =
(\upsilon^{\alpha})$ we have

\[ \frac{\delta \upsilon^{\alpha}}{dt} = [h(t)]^{\prime}\dot{z}^{\alpha} + [l(t)]^{\prime}\tilde{\dot{z}}^{\alpha}  \ ,\]

which shows that a random vector from the holomorphic vector space
$\{ \dot{z}^{\alpha}, \tilde{\dot{z}}^{\alpha} \}$, at a parallel
transfer with respect to $z^{\alpha} (t)$, remains in $\{
\dot{z}^{\alpha}, \tilde{\dot{z}}^{\alpha} \}$, i.e. $z^{\alpha}
(t)$ is a PH-curve.
\end{proof}

\begin{assertion}
If $z^{\alpha} (t)$ is geodesic for $\mathfrak{M}(f)$ and
$\dot{z}^{\alpha} \in \ker f$, then $z^{\alpha} (t)$ is a PH-curve
and it belongs to the special plane (straight line) $\omega$ in the
biaxial space $B_{n}(B_{3})$.
\end{assertion}

\begin{proof}
The assertion is obvious because
\[ \frac{\delta}{dt}
\tilde{\dot{z}}^{\alpha} = 0 \ . \]
\end{proof}

\begin{assertion}
If $z^{\alpha} (t)$ is a PH-curve and $\dot{z}^{\alpha} \in \ker f$,
then $z^{\alpha} (t)$ is geodesic.
\end{assertion}

\begin{proof}
From the differential equation of PH-curves follows that
\[ \frac{\delta \dot{z}^{\alpha}}{dt} = a(t) \dot{z}^{\alpha} \ . \]
The tangential direction is transferred parallely with respect to
$z^{\alpha} (t)$, therefore $z^{\alpha} (t)$ is geodesic.
\end{proof}

\

We note down that in this case all curves
\[ z^{\alpha} (t) = (c_{1}, \dots , c_{n}, z^{n+1}(t), \dots , z^{2n} (t), z^{2n+1} (t), \dots , z^{2n+m} (t)) \]
possess one and the same tangential vector field. The special case,
where $n = 1, m = 0$ in the two-dimensional affine plane $0xy$, is
interesting. There $z^{\alpha} (t) = (c, z(t))$. The corresponding
geodesic curve is a straight line, which is parallel to the axis
$0y$.

\

Let us consider a particular case, where $\mathfrak{M}(f)$ is a
tangential differentiation of the $B$ manifold.

\

\begin{definition}
We say that the pair of functions $p(u,v), q(u,v)$ is holomorphic if
\[ \frac{\partial p}{\partial u} = \frac{\partial q}{\partial v}\]
and
\[ \frac{\partial p}{\partial v} = 0 \ . \]
\end{definition}

\

On the basis of this definition of the two-dimensional surface
\[ S^{H} \subset \mathfrak{M}(f) : \{ z^{i}(u); z^{n+i}(u,v) = v\frac{dz^{i}}{du} \} \ , \]
we will say that it is holomorphic in $\mathfrak{M}(f)$.

\begin{assertion}
As regards a holomorphic two-dimensional surface $S^{H}$, with the
change of the parameters $\overline{u} = h(u)$, $\overline{v} =
t(u,v)$ in such a way that for the differentiable functions $h(u)$
and $t(u,v)$ there holds
\[ \overline{u} = h(u) \]
and
\[ \frac{dh}{du} = \frac{\partial t}{\partial v} \ , \]
the holomorphicity of $S^{H}$ is preserved.
\end{assertion}

\begin{proof}
We substitute the values in the parametric equation of
\[ S^{H} : z^{i} = z^{i}(\overline{u}), z^{n+i} = \overline{v}\frac{dz^{i}}{d\overline{u}} (\overline{u}, \overline{v}) \]
with their equals. In this case $z^{i}$ depend only on the variables
$u$, and at that
\[ \frac{dz^{i}}{du} = \frac{dz^{i}}{d\overline{u}} . \frac{dh}{du} \ ; \]

\

\[ \frac{\partial z^{n+i}}{\partial v} = \frac{\partial z^{n+1}}{\partial \overline{v}} . \frac{dt}{dv} = \frac{\partial}{\partial \overline{v}} (\overline{v} \frac{dz^{i}}{d\overline{u}}) . \frac{\partial t}{\partial v} = \frac{dz^{i}}{d\overline{u}} . \frac{dh}{du} = \frac{dz^{i}}{du} \ , \ \ \mbox{i.e.}\]

\

\[ \frac{\partial z^{n+i}}{\partial v} = \frac{dz^{i}}{du} \ .\]
\end{proof}

\

\begin{theorem}
Every geodesic curve $\beta(u) = (u = u, v = const)$ ($u$ - line)
from the two-dimensional holomorphic surface $S^{H} \subset
\mathfrak{M}(f)$ is a PH-curve for $\mathfrak{M}(f)$.
\end{theorem}

\begin{proof}
Let $\gamma (v) = (u = const, v = v)$ be the $v$-lines for $S^{H}$,
and
\[ \dot{\beta} = (\dot{z}^{i}; v\ddot{z}^{i}; 0, \dots , 0) \]
and
\[ \dot{\gamma} = (0, \dots ; \dot{z}^{i}; 0, \dots , 0)  \]
are the corresponding velocities of $\beta(u)$ and $\gamma(u)$.
Obviously
\[ \dot{\gamma} = f\dot{\beta} \ . \]
Besides,
\[
\frac{\delta \dot{\beta}^{\alpha}}{du} = \frac{d
\dot{\beta}^{\alpha}}{du} + \mathfrak{g}^{\alpha}_{\lambda \nu}
\dot{\beta}^{\lambda} \dot{\beta}^{\nu} = 0 \ .
\]
Taking into consideration the purity of the objects
$\mathfrak{g}^{\alpha}_{\lambda \nu}$ from the last two equations,
we obtain
\begin{eqnarray}
\frac{\delta \dot{\gamma}^{\alpha}}{du} &=& \frac{d}{du}
(f^{\alpha}_{\lambda} \dot{\beta}^{\lambda}) +
\mathfrak{g}^{\alpha}_{\nu \sigma} f^{\nu}_{\lambda}
\dot{\beta}^{\rho} \dot{\beta}^{\sigma} = \nonumber \\ \ \nonumber \\
&=& \frac{d}{du} (f^{\alpha}_{\lambda} \dot{\beta}^{\lambda}) +
\mathfrak{g}^{\lambda}_{\rho \sigma} f^{\alpha}_{\lambda}
\dot{\beta}^{\rho} \dot{\beta}^{\sigma} = \nonumber \\ \ \nonumber
\\ &=& \Big[ \frac{d}{du} \dot{\beta}^{\lambda} + \mathfrak{g}^{\lambda}_{\rho \sigma} \dot{\beta}^{\rho} \dot{\beta}^{\sigma} \Big] f^{\alpha}_{\lambda} = 0 \nonumber
\end{eqnarray}
Therefore $\dot{\gamma}$ is transferred parallely with respect to
the $u$-lines $\beta$, i.e. $\beta$ is a PH-curve for
$\mathfrak{M}(f)$.
\end{proof}

\

We note down that the theorem proved above is a confirmation of
Assertion 2. Here the object of consideration is a holomorphic
two-dimensional surface $S^{H}$ from $\mathfrak{M}(f)$, and we can
always consider that on it there is set a geodesic vector field,
which does not belong to $\ker f$. As it is recorded in \cite{9},
$S^{H}$ is interpreted as a real model of a geodesic curve on the
manifold $\overset{\ast}{\mathfrak{M}}$ over the algebra
\[ \mathbb{R}(\varepsilon) = \{a + \varepsilon b; \varepsilon^{2} = 0; a,b \in \mathbb{R} \} \ . \]
There the surfaces $S^{H}$ are called analytical.

\

If $q=q_{i}$ is a form on $\mathfrak{M}(f)$, there can be found a
lot of connections, which are pure with respect to $f$. The tensor
of the affine deformation $T$, through which these connections are
obtained, is
\[ T^{n}_{ik} = \delta^{n}_{i} \tilde{q}_{k} + f^{n}_{i} q_{k} + \delta^{n}_{k} \tilde{q}_{i} + f^{n}_{k} q_{i} \ . \]
In the above equation $\tilde{q}_{i} := q_{s} f^{s}_{i}$.

\

Obviously $T$ is pure with respect to all indices. Therefore any
connection
\[ \overline{\Gamma}^{\alpha}_{\beta \gamma} = \Gamma^{\alpha}_{\beta \gamma} + T^{\alpha}_{\beta \gamma} \]

is pure too.

\begin{theorem}
The connection $\overline{\nabla}$ with coefficients
$\overline{\Gamma}^{\alpha}_{\beta \sigma}$ possesses the following
properties:

(1) it has symmetrical and pure coefficients,

(2) $\overline{\nabla} f = 0 \ ,$

(3) $\overline{\nabla}$ and $\nabla$ have common PH-curves.

\end{theorem}

\begin{proof}

\

\

\emph{(1)} The integrated tensor of $T$ is
\[ \overline{T}^{n}_{is} = T^{n}_{in} f^{n}_{s}= f^{n}_{i}\tilde{q}_{s} + f^{n}_{s}\tilde{q}_{i} \ . \]

It hence follows that the connection with coefficients
$\overline{\Gamma}^{\alpha}_{\beta \gamma}$ is also pure.

\

\emph{(2)} Let us assume that $\nabla f = 0$. Then
\begin{eqnarray}
\overline{\nabla} f^{n}_{i} &=& \overline{\Gamma}^{n}_{sj} f^{j}_{i}
- \overline{\Gamma}^{p}_{si} f^{n}_{p} = \nonumber \\ \ \nonumber \\
&=& \Gamma^{n}_{sj} f^{j}_{i} + T^{n}_{sj} f^{j}_{i} -
\Gamma^{p}_{si} f^{n}_{p} - T^{p}_{si} f^{n}_{p} = \nonumber \\ \
\nonumber \\ &=& \nabla f + \tilde{T}^{n}_{si} - \tilde{T}^{n}_{si}
= 0 \ . \nonumber
\end{eqnarray}

\

\emph{(3)} We denote with $\overline{\delta}$ the covariant
differential as regards $\overline{\nabla}$. Since
\[ \frac{\delta}{dt}\dot{x}^{n} = a(t)\dot{x}^{n} + b(t)\tilde{\dot{x}}^{n} \ , \]

there follows that
\begin{eqnarray}
\frac{\overline{\delta}}{dt}\dot{x}^{n} &=& \ddot{x} +
[\Gamma^{n}_{ik} + \delta^{n}_{i} \tilde{p}_{k} + \delta^{n}_{k}
\tilde{p}_{i} + f^{n}_{i}q_{k} + f^{n}_{k}q_{i} ] \dot{x}^{i}
\dot{x}^{n} = \nonumber \\ \ \nonumber \\ &=&
\frac{\delta}{dt}\dot{x}^{n} + \dot{x}^{n} \tilde{p}(\dot{x}) +
\dot{x}^{n} \tilde{p}(\dot{x}) + \tilde{\dot{x}}^{n}q(\dot{x}) +
\tilde{\dot{x}}^{n}q(\dot{x}) \nonumber \ .
\end{eqnarray}
Here we have designated $p_{k} := \tilde{q}_{k}$. Finally we obtain
\[ \frac{\overline{\delta}}{dt}\dot{x}^{n} = (a + 2\tilde{p}(\dot{x}))\dot{x}^{n} + (b + 2q(\dot{x}))\tilde{\dot{x}}^{n} \ . \]

\end{proof}
\begin{center}

\end{center}
{\small
\begin{tabular}{lcl}
\\
\\
Asen Hristov \\
Faculty of Mathematics and Informatics\\
University of Plovdiv "P. Hilendarski"\\
24, Tzar Assen Str.\\
4000 Plovdiv, BULGARIA\\
E-mail: hristov-asen@mail.bg&{}&
\end{tabular}}
\end{document}